\documentclass[11pt]{amsart}

\usepackage[margin=1in]{geometry}
\usepackage{setspace}
\usepackage{color}
\usepackage{cancel}
\usepackage{hyperref}

\newtheorem{theorem}{Theorem}[section]
\newtheorem{lemma}[theorem]{Lemma}
\newtheorem{proposition}[theorem]{Proposition}
\newtheorem{corollary}[theorem]{Corollary}
\newtheorem{conjecture}[theorem]{Conjecture}

\theoremstyle{definition}

\newtheorem{example}[theorem]{Example}

\theoremstyle{remark}
\newtheorem{remark}[theorem]{Remark}

\newcommand{\HH}{\mathrm{H}} \newcommand{\TT}{\mathrm{T}}

\begin{document}

\title{A coin flip game and generalizations of Fibonacci numbers}
\author{Jia Huang}
\address{Department of Mathematics and Statistics, University of Nebraska at Kearney, Kearney, NE 68849, USA}
\email{huangj2@unk.edu}



\begin{abstract}
We study a game in which one keeps flipping a coin until a given finite string of heads and tails occurs.
We find the expected number of coin flips to end the game when the ending string consists of at most four maximal runs of heads or tails or alternates between heads and tails.
This leads to some summation identities involving certain generalizations of the Fibonacci numbers.
\end{abstract}

\maketitle

\section{Introduction}\label{sec:intro}

In an X-post on March 16 2024, Daniel Litt~\cite{Litt} asked the following question (see also Cheplyaka~\cite{Cheplyaka}): If you flip a coin $100$ times, will HT (a heads followed by a tails) occur more often than HH (two heads in a row) or the other way around?
Ekhad and Zeilberger~\cite{EkhadZeilberger} studied this and some more general questions by symbolic computation.
Segert~\cite{Segert} proved that it is more likely for HT to outnumber HH when a coin is flipped at least three times and provided certain asymptotic results.
Basdevant, H\'enard, Maurel-S\'egala, and Singh~\cite{BHMS} generalized this game by allowing two players choosing two different strings of heads and tails with the same length.

While all the above research work and other related online discussions occurred in 2024, Simonson discussed in a book published in 2022~\cite[Ch.~2]{Simonson} a somewhat similar (but different) game in which a carnie offers \$5 to you while in a carnival and asks you to pay \$$n$ back, where $n$ is the number of times you flip a coin until you have two consecutive heads. 
Let $E$ be the expected value of coin flips to end this game.
If the first flip gives a tail, then the expected value of flips to get two heads in a row will be $1+E$.
If the first flip gives a heads and the next is a tails, then the expected value of flips will be $2+E$.
If the first two flips both give a heads then the expected flips is $2$.
It follows that
$E = \frac12(1+E) + \frac14(2+E) + \frac 2 4$.
This implies $E=6$, so in average you will lose \$1 each time when you play this game.
 
There is another way to compute $E$.
Let $E_n$ be the number of outcomes from flipping a coin $n$ times such that the game will end right after $n$ flips.
This turns out to be the Fibonacci number $F_{n-1}$ since $E_n=n-1$ for $n=1,2$ and for $n\ge3$, if the first flip is a heads, then the next flip must be a tails, so by induction, $E_n = F_{n-3} + F_{n-2} = F_{n-1}$.
It follows that the expected value of $n$ is $\sum_{n=1}^\infty n F_{n-1}/2^n$, which must equal $6$.
This can also be proved by standard generating function technique (see, e.g., Vajda~\cite{Vajda}), although Simonson did not address that, perhaps due to the scope of his book~\cite{Simonson}.
More generally, Benjamin, Neer, Otero, and Sellers~\cite{ProbFibSum} obtained some interesting Fibonacci sums by computing the $n$th power of the expected number of coin flips to obtain $\HH\HH$.

Simonson~\cite[Ch.~2]{Simonson} also mentioned that, since you will lose money eventually by playing the first game a number of times, what if you are offered to switch roles with the carnie but using HT instead of HH to end the game?
Similarly as the above, one can show that the expected number of flips to end this game is $E=4$, so you should not get tricked again by the carnie.
Interestingly, as mentioned by Simonson~\cite[Ch.~2]{Simonson}, although the expected number of coin flips to get HH is more than the expected number of coin flips to get HT (and thus it is more likely for HT to outnumber HH in a long run), it is equally likely to get HT or HH when either occurs, since it is a $1/2$ chance to get either a heads or a tails after a heads.

In this paper, we study a more general coin flip game which will be terminated right after the first occurrence of a given finite string $S$ of heads and tails.
We call $S$ the \emph{ending string} of this game and let $E(S)$ denote the expected number of flips to end the game.
For any integers $k,\ell\ge1$, we show that $E(\HH^k) = 2^{k+1}-2$ and
$E(\HH^k\TT^\ell)= 2^{k+\ell}$.
This result implies not only that $E(\HH\HH)=6$ and $E(\HH\TT)=4$ as mentioned earlier but also the following identities:
\[ \sum_{n=0}^\infty \frac{n F_{n-1}^k}{2^n} = 2^{k+1}-2 \quad\text{and}\quad
\sum_{n=0}^\infty \frac{n \overline F_{n-2}^k}{2^n} = 2^{k+1}. \]
Here $F_n^k$ is the \emph{Fibonacci number of order $k$} and $\overline F_n^k$ is a variation of $F_n^k$ (see Section~\ref{sec:id}).
We also show that $E(\HH^k\TT^\ell\HH^m) = 2^{k+\ell+m} + 2^{\min\{k,m\}+1} -2$ for any integers $k,\ell,m\ge1$ and derive a summation identity involving a two-parameter generalization of both $F_n^k$ and $\overline F_n^k$.
Moreover, we prove that for any integer $k,\ell,m,d\ge1$,
\[ E(\HH^k\TT^\ell\HH^m\TT^d) = \begin{cases}
2^{k+\ell+m+d}, & \text{if } m<k \text{ or } d>\ell; \\
2^{k+\ell+m+d} + 2^{k+d}, & \text{if } m\ge k \text{ and } d\le\ell.
\end{cases} \]
This leads to a summation identity involving a different two-parameter generalization of the Fibonacci numbers.

Thanks to the symmetry between heads and tails, we have settled all the cases when the ending string of the coin flip game has at most four maximal runs of heads or tails, although it would need a different type of symmetry argument (see Section~\ref{sec:questions}) to explain the equalities $E(\HH^k\TT^\ell) = E(\TT^\ell\HH^k)$, $E(\HH^k\TT^\ell\HH^m) = E(\HH^m\TT^\ell\HH^k)$, and $E(\HH^k\TT^\ell\HH^m\TT^d) = E(\HH^d\TT^m\HH^\ell\TT^k)$, all implied by our results mentioned above.
Lastly, for ending strings alternating between heads and tails, we obtain 
\[
E((\HH\TT)^k) = \sum_{i=1}^k 2^{2i} = \frac{2^{2k+2}-4}{3}
\quad\text{and}\quad
E((\HH\TT)^{k-1}\HH) = \sum_{i=1}^k 2^{2i-1} = \frac{2^{2k+1}-2}{3}
\]
and derive some summation identities as well.

This paper is structured as follows. 
In Section~\ref{sec:main} we determine $E(S)$ when $S$ has at most four maximal runs of heads or tails or alternates between heads and tails.
In Section~\ref{sec:id} we derive summation identities from our results in the last section.
Finally, we ask some questions for future research in Section~\ref{sec:questions}.

\section{Expected number of coin flips}\label{sec:main}

Let $S$ be the ending string of the coin flip game described in Section~\ref{sec:intro}.
To study the expected number $E(S)$ of flips to end this game, we need the identities below, whose proofs are included to be complete.

\begin{lemma}\label{lem:id}
For any integer $k\ge1$, 
we have $\displaystyle \sum_{i=1}^k \frac{1}{2^i} = 1-\frac{1}{2^k}$ and 
$\displaystyle \sum_{i=1}^k \frac{i}{2^i} = 2-\frac{k+2}{2^k}$.
\end{lemma}
\begin{proof}
Substituting $x=\frac12$ into the well-known formula 
$\sum_{i=0}^k x^i = (1-x^{k+1})/(1-x)$
gives the first identity.
Taking the derivative of both sides of the above formula,
then multiplying both sides by $x$ and substituting in $x=\frac12$ we obtain the second desired identity.
\end{proof}

We first determine $E(S)$ when $S$ consists of only one or two maximal runs of heads or tails.
By the symmetry between heads and tails, we may assume that $S$ begins with a heads.

\begin{theorem}\label{thm1}
For any integers $k,\ell\ge1$, we have $E(\HH^k) = 2^{k+1}-2$ and $E(\HH^k\TT^\ell) = 2^{k+\ell}$.
\end{theorem}

\begin{proof}
First let $E=E(S)$, where $S=\HH^k$. 
If the $i$th flip gives the first tails, then we must have $i\le k$, and the expected number of flips to end the game will be $E(S|\HH^{i-1}\TT) = E+i$ in this case; otherwise the first $k$ flips are all heads, and $E(S|\HH^k)=k$.
It follows that
\[ E = \sum_{i=1}^k \frac{E+i}{2^i} + \frac{k}{2^k}
= E -\frac{E}{2^k} + 2-\frac{k+2}{2^k} + \frac{k}{2^k}, 
\]
where the second equality holds by Lemma~\ref{lem:id}.
This implies that $E = 2^{k+1}-2$.

Now let $E=E(S)$, where $S=\HH^k\TT^\ell$.
Then $E=\frac12 E(S|\HH) + \frac12 E(S|\TT)$ and $E(S|\TT) = E+1$, so $E(S|\HH) = E-1$.
Suppose the first flip gives a heads.
Then the $(i+1)$th flip gives the first tails for some integer $i\ge1$.
If $i<k$ then $E(S|\HH^i\TT) = E+i+1$.
Assume $i\ge k$ below.
After the first $i$ flips (all heads) there will be a maximal run of $j$ tails for some integer $j\in\{1,2,\ldots,\ell\}$.
If $j<\ell$ then the next flip is a heads and $E(S|\HH^i\TT^j\HH) = i+j+E(S|\HH) = E+i+j-1$.
If $j=\ell$ then the game ends and $E(S|\HH^i\TT^\ell)=i+\ell$.
Thus
\begin{align*}
\frac{E-1}{2} &= \sum_{i=1}^{k-1} \frac{E+i+1}{2^{i+1}} + 
\sum_{i=k}^\infty \left( \sum_{j=1}^{\ell-1} \frac{E+i+j-1}{2^{i+j+1}} + \frac{i+\ell}{2^{i+\ell}} \right) \\
&= \frac{E}{2} - \frac{E}{2^k} + \frac{3}{2} - \frac{k+2}{2^{k}} + 
\sum_{i=k}^\infty \left( \frac{1}{2^i} \left( \frac{E}{2} - \frac{E}{2^\ell} \right) 
+ \frac{i+1}{2^{i+1}} - \frac{i+\ell}{2^{i+\ell}} + \frac{i+\ell}{2^{i+\ell}} \right) \\
&= \frac{E}{2} - \frac{E}{2^k} + \frac{3}{2} - \frac{k+2}{2^{k}} + \frac{E}{2^k} - \frac{E}{2^{k+\ell-1}} + \frac{k+2}{2^{k}}.
\end{align*}
Here the second and third equalities hold by Lemma~\ref{lem:id}.
It follows that $E = 2^{k+\ell}$.
\end{proof}

To tackle the cases when $S$ consists of more than two maximal runs of heads or tails, we define $E(S|R)$ to be the expected number of flips to obtain $S$ given that the first $r$ flips produce $R$, where $R$ is a string of heads and tails with length $r$, and develop some lemmas.

\begin{lemma}\label{lem:RS}
If the ending string $S$ of the coin flip game begins with a string $R$ of length $r$, then $E(S) = E(R) + E(S|R) - r$.
\end{lemma}

\begin{proof}
If $S$ occurs then $R$ must occur first.
It takes $E(R)$ flips in average to obtain the first occurrence of $R$, and all the flips before that will not affect when the game ends.
It will then take $E(S|R)-r$ extra flips in average to obtain $S$.
Thus $E(S) = E(R) + E(S|R) - r$.
\end{proof}

\begin{lemma}\label{lem:HT}
The following holds for $E(S)$, where $S$ begins with a maximal run of $k\ge1$ heads.
\begin{itemize}
\item
If $i<k$ then $E(S|\HH^i\TT) = E+i+1$ and $E(S|\HH^i) = E+i+2- 2^{i+1}$.
\item
If $i\ge k$ then 
$E(S|\HH^i\TT) = E+i+1-2^{k+1}$ and $E(S|\HH^i) = E+i+2-2^{k+1}$.
\end{itemize}
\end{lemma}

\begin{proof}
Let $E=E(S)$.
For $i<k$, if the first $i+1$ flips give $\HH^i\TT$ then the game is reset, so $E(S|\HH^i\TT) = E+i+1$. 
For $i\le k$, Lemma~\ref{lem:RS} implies that 
\[ E(S|\HH^i) = E + i - E(\HH^i) = E + i + 2 - 2^{i+1}, \]
which can also be proved by using $E(S|\HH^i) = 2E(S|\HH^{i-1}) - E(S|\HH^{i-1}\TT)$ and induction on $i$.
For $i>k$ we have 
\[ E(S|\HH^i) = i-k+E(S|\HH^k) = E+i+2-2^{k+1}. \]
Then for $i\ge k$ we have
\[ E(S|\HH^i\TT) = 2E(S|\HH^i)-E(S|\HH^{i+1}) 
= 2(E+i+2-2^{k+1}) - (E+i+3-2^{k+1}) = E+i+1-2^{k+1}. 
\qedhere \]
\end{proof}

Before any new result, we use the above lemmas to give a more concise proof for Theorem~\ref{thm1}.

\begin{proof}[Another proof for Theorem~\ref{thm1}]
First let $E=E(S)$, where $S=\HH^k$ for some $k\ge1$.
By Lemma~\ref{lem:HT}, $E(S|\HH^{k-1}) = E+k+1-2^k$.
Also, $E(S|\HH^{k-1}) = \frac12(E+k) + \frac{k}{2} = \frac12 E + k$.
Thus $E = 2^{k+2}-2$.

Now let $E=E(S)$, where $S=\HH^k\TT^\ell$ for some integers $k,\ell\ge1$.
If $\ell=1$ then by Lemma~\ref{lem:HT}, $k+1=E(S|\HH^k\TT) = E+k+1-2^{k+1}$, so $E=2^{k+1}$.
Assume $\ell\ge2$ below. 
By Lemma~\ref{lem:RS}, $E = E(\HH^k\TT^{\ell-1}) + E(S|\HH^k\TT^{\ell-1}) -k-\ell+1$. 
Since $E(S|\HH) = E-1$ by Lemma~\ref{lem:HT}, we have $E(S|\HH^k\TT^{\ell-1}) = \frac12(E+k+\ell-2) + \frac12(k+\ell)$.
Thus $E=2E(\HH^k\TT^{\ell-1})$.
By induction, $E=2^{k+\ell}$.
\end{proof}

Similarly to Theorem~\ref{thm1}, we now determine $E(S)$ when $S$ consists of three maximal runs of heads or tails and provide two proofs for our result.

\begin{theorem}\label{thm2}
For $k,\ell,m\ge1$ we have $E(\HH^k\TT^\ell\HH^m) = 2^{k+\ell+m} + 2^{\min\{k,m\}+1} -2$.
\end{theorem}

\begin{proof}[Proof 1]
If $S=\HH^k\TT^\ell\HH^m$ occurs, then $S'=\HH^k\TT^\ell$ must occur first.
It takes $E(S')$ flips in average to obtain $S'$.
If the next flip is tails, then the game is reset and it will still take in average $E=E(S)$ more flips after that tails to obtain $S$.
Otherwise the first occurrence of $S'$ is followed by a maximal run of $i$ heads for some integer $i\in\{1,2,\ldots,m\}$.
If $i=m$ then the game ends right away.
Assume $i<m$ below, so the first occurrence of $S'=\HH^k\TT^\ell$ will not affect when the game ends.
If $i<k$ then the game is reset by $\HH^i\TT$ and it will still take extra $E$ flips in average to obtain $S$.
If $k\le i < m$ then we have $E(\HH^i\TT) = E+i+1-2^{k+1}$ by Lemma~\ref{lem:HT}.
Thus
\begin{align*}
E &= E(S') + \frac{E+1}{2} + \sum_{i=1}^{\min\{k,m\}-1} \frac{E+i+1}{2^{i+1}} + \sum_{i=\min\{k,m\}}^{m-1} \frac{E+i+1-2^{k+1}}{2^{i+1}} + \frac{m}{2^m} \\
&= E(S') + \frac{E+1}{2} + \frac{E}{2} - \frac{E}{2^m} + \frac{3}{2} - \frac{m+2}{2^m} + \frac{2^{k+1}}{2^{m}} - \frac{2^{k+1}}{2^{\min\{k,m\}}} + \frac{m}{2^m}.
\end{align*}
Here the second equality follows from Lemma~\ref{lem:id}.
Then we can solve for $E$ from the above.
\end{proof}

\begin{proof}[Proof 2]
Let $S=\HH^k\TT^\ell\HH^m$.
By Lemma~\ref{lem:RS} and Lemma~\ref{lem:HT},
\begin{align*}
E &= E(\HH^k\TT^\ell\HH^{m-1}) + E(S|\HH^k\TT^\ell\HH^{m-1}) -k-\ell-m+1 \\
&= E(\HH^k\TT^\ell\HH^{m-1}) + \frac12(k+\ell+m) + \frac12 E(S|\HH^k\TT^\ell\HH^{m-1}\TT) -k-\ell-m+1.
\end{align*}

First assume $m\le k$.
Then $E(S|\HH^k\TT^\ell\HH^{m-1}\TT) = E+k+\ell+m$.
Thus $E=2E(\HH^k\TT^\ell\HH^{m-1}) + 2$.

Assume $m>k$ below.
Using $E(S|\HH^k\TT) = E+k+1-2^{k+1}$ from Lemma~\ref{lem:HT}, we obtain 
\[ E(S|\HH^k\TT^\ell\HH^{m-1}\TT) = E(S|\HH^k\TT)+\ell+m-1 
= E+k+\ell+m-2^{k+1}.
\]
It follows that $E=2E(\HH^k\TT^\ell\HH^{m-1}) + 2 - 2^{k+1}$.

Applying the above recursion to $E(\HH^k\TT^\ell)=2^{k+\ell}$ from Theorem~\ref{thm1} finishes the proof.
\end{proof}

\begin{example}
Theorem~\ref{thm1} implies $E(\HH\HH)=E(\TT\TT)=2^3-2=6$ and $E(\HH\TT)=E(\TT\HH)=2^2=4$, both mentioned in Section~\ref{sec:intro}.
Applying Theorem~\ref{thm1} and Theorem~\ref{thm2} to all ending strings of length three, we have $E(\HH\HH\HH)=E(\TT\TT\TT)=2^4-2=14$, which is much larger than $E(\HH\HH\TT) = E(\TT\TT\HH) = E(\HH\TT\TT) = E(\TT\HH\HH) =2^3=8$ (although when HHH or HHT occurs, it is a 50\% chance to get either one), and $E(\HH\TT\HH)=E(\TT\HH\TT)=2^3+2^2-2=10$ lies in between.
\end{example}

Next, we establish a formula for $E(S)$ when $S$ consists of four maximal runs of heads or tails.
We still provide two proofs, both based on the same strategy as in the second proof of Theorem~\ref{thm2}, hoping that at least of them could be helpful in future studies of more complicated cases.

\begin{theorem}\label{thm3}
Let $E=E(S)$, where $S=\HH^k\TT^\ell\HH^m\TT^d$ with $k,\ell,m,d\ge1$.
Then \[ E = \begin{cases}
2^{k+\ell+m+d}, & \text{if } m<k \text{ or } d>\ell; \\
2^{k+\ell+m+d} + 2^{k+d}, & \text{if } m\ge k \text{ and } d\le\ell.
\end{cases} \]
\end{theorem}

\begin{proof}[Proof 1]
We start to compute $E$ with the help of Lemma~\ref{lem:RS}:
\begin{align*} 
E &= E(\HH^k\TT^\ell\HH^m) + E(S|\HH^k\TT^\ell\HH^m) -k-\ell-m \\ 
&= E(\HH^k\TT^\ell\HH^m) + \frac{E(S|\HH^k\TT^\ell\HH^{m+1})}2 + \frac{E(S|\HH^k\TT^\ell\HH^m\TT^d)}{2^d}
+ \sum_{i=1}^{d-1} \frac{E(S|\HH^k\TT^\ell\HH^m\TT^i\HH)}{2^{i+1}} -k-\ell-m.
\end{align*}
We have $E(\HH^k\TT^\ell\HH^m)=2^{k+\ell+m} + 2^{\min\{k,m\}+1} -2$ by Theorem~\ref{thm2}.
Lemma~\ref{lem:HT} implies
\[ E(S|\HH^k\TT^\ell\HH^{m+1}) = E(S|\HH^{m+1})+k+\ell = 
\begin{cases}
E+k+\ell+m+3-2^{m+2}, & \text{if } m+1<k; \\
E+k+\ell+m+3-2^{k+1}, & \text{if } m+1\ge k.
\end{cases} \]
It is clear that $E(S|\HH^k\TT^\ell\HH^m\TT^d)=k+\ell+m+d$.
By Lemma~\ref{lem:HT}, $E(S|\HH) = E-1$ and
\begin{align*}
E+k+1-2^{k+1} &= E(S|\HH^k\TT) = \sum_{i=1}^\infty \frac{E(S|\HH^k\TT^i\HH)}{2^i} \\
&= \sum_{i\ge 1,\ i\ne \ell} \frac{E-1+k+i}{2^i} + \frac{E(S|\HH^k\TT^\ell\HH)}{2^\ell} \\
&= E+k-1 - \frac{E+k-1}{2^\ell} + 2 - \frac{\ell}{2^\ell} + \frac{E(S|\HH^k\TT^\ell\HH)}{2^\ell}.
\end{align*}
Thus $E(S|\HH^k\TT^\ell\HH) = E + k + \ell - 1 - 2^{k+\ell+1}$.
Then for $i=1,\ldots,d-1$ we have 
\[ E(S|\HH^k\TT^\ell\HH^m\TT^i\HH) = 
\begin{cases}
E+k+\ell+m+i-1, & \text{if } m<k \text{ or } i\ne \ell; \\
E+k+\ell+m+i-1-2^{k+\ell+1}, &\text{if } m\ge k \text{ and } i=\ell.
\end{cases} \]
Therefore, we can solve for $E$ from the above and obtain the desired result.
\end{proof}

\begin{proof}[Proof 2] 
By Lemma~\ref{lem:RS}, we have
\begin{align*} 
E &= E(\HH^k\TT^\ell\HH^m\TT^{d-1}) + E(S|\HH^k\TT^\ell\HH^m\TT^{d-1}) -k-\ell-m-d+1 \\ 
&= E(\HH^k\TT^\ell\HH^m\TT^{d-1}) + \frac12 E(S|\HH^k\TT^\ell\HH^m\TT^{d-1}\HH) - \frac12(k+\ell+m+d) +1,
\end{align*}
where we use $E(S|\HH^k\TT^\ell\HH^m\TT^{d-1}) = \frac12(k+\ell+m+d) + \frac12 E(S|\HH^k\TT^\ell\HH^m\TT^{d-1}\HH)$.

First assume $d=1$. 
Lemma~\ref{lem:HT} implies
\[ E(S|\HH^k\TT^\ell\HH^m\TT^{d-1}\HH) = E(S|\HH^{m+1})+k+\ell = 
\begin{cases}
E+k+\ell+m+3-2^{m+2}, & \text{if } m+1<k; \\
E+k+\ell+m+3-2^{k+1}, & \text{if } m+1\ge k.
\end{cases} \]
We also have $E(\HH^k\TT^\ell\HH^m\TT^{d-1}) = E(\HH^k\TT^\ell\HH^m)= 2^{k+\ell+m} + 2^{\min\{k,m\}+1}-2$ by Theorem~\ref{thm2}.
Thus 
\[ E 
= \begin{cases}
2^{k+\ell+m+1}, & \text{if } m<k; \\
2^{k+\ell+m+1} + 2^{k+1}, & \text{if } m \ge k.
\end{cases} \]

Now assume $d\ge 2$. 
We have $E(S|\HH^k\TT^\ell\HH) = E+k+\ell-1-2^{k+\ell+1}$ from the previous proof and $E(S|\HH) = E-1$ by Lemma~\ref{lem:HT}.
Thus
\[ E(S|\HH^k\TT^\ell\HH^m\TT^{d-1}\HH) = 
\begin{cases}
E+k+\ell+m+d-2-2^{k+\ell+1}, & \text{if } m\ge k \text{ and } d=\ell+1; \\
E+k+\ell+m+d-2, & \text{if } m<k \text{ or } d\ne\ell+1.
\end{cases} \]
Then we can solve for $E$ and obtain 
\[ E = \begin{cases}
2E(S|\HH^k\TT^\ell\HH^m\TT^{d-1}) -2^{k+d}, & \text{if } m\ge k \text{ and } d=\ell+1;\\ 
2E(S|\HH^k\TT^\ell\HH^m\TT^{d-1}), & \text{if } m<k \text{ or } d\ne \ell+1.
\end{cases}\]
Applying this recursive relation to $E(\HH^k\TT^\ell\HH^m\TT)$ gives the desired result.
\end{proof}

Finally, we find $E(S)$ when $S$ is a finite string alternating between heads and tails.

\begin{theorem}\label{thm:alternate}
For any integer $k\ge1$, we have 
\[ 
E((\HH\TT)^k) = \sum_{i=1}^k 2^{2i} = \frac{2^{2k+2}-4}{3}
\quad\text{and}\quad
E((\HH\TT)^{k-1}\HH) = \sum_{i=1}^k 2^{2i-1} = \frac{2^{2k+1}-2}{3}.
\]
\end{theorem}

\begin{proof}
First let $E=E(S)$, where $S=(\HH\TT)^k$.
By Lemma~\ref{lem:RS}, we have
\begin{align*}
E &= E((\HH\TT)^{k-1}\HH) + E(S|(\HH\TT)^{k-1}\HH) - 2k + 1 \\
&= E((\HH\TT)^{k-1}\HH) + \frac12(2k-1+E(S|\HH) + 2k) - 2k + 1.
\end{align*}
Since $E(S|\HH) = E-1$ by Lemma~\ref{lem:HT}, it follows that $E=2E((\HH\TT)^{k-1}\HH)$.

Now let $E=E(S)$, where $S=(\HH\TT)^{k-1}\HH$.
Then $E=2E((\HH\TT)^{k-1})+2$ since Lemma~\ref{lem:RS} implies
\begin{align*}
E &= E((\HH\TT)^{k-1}) + E(S|(\HH\TT)^{k-1}) - 2k + 2 \\
&= E((\HH\TT)^{k-1}) + \frac12(2k-1 + E + 2k-1) - 2k + 2.
\end{align*}

Applying the above recursive relations to the initial condition $E(\HH) = 2$, which follows from the observation that $E(\HH) = \frac12 + \frac12 (E(\HH)+1)$, gives the desired result.
\end{proof}

\section{Generalizations of Fibonacci numbers}
\label{sec:id}

In this section we derive some identities from our results in Section~\ref{sec:main} using the observation that
\begin{equation}\label{eq:E_n}
E(S) = \sum_{n=0}^\infty \frac{n E_n(S)}{2^n}. 
\end{equation}
Here $E_n(S)$ denotes the number of strings consisting of a total of $n$ heads and tails such that $S$ occurs only at the end.
Also define $E_n(S|R)$ to be the number of strings consisting of a total of $n$ heads and tails with $R$ occurring at the beginning (possibly elsewhere) and $S$ occurring only at the end.
The following lemma will be useful when we compute $E_n(S)$.

\begin{lemma}\label{lem:HT'}
Suppose $k,\ell,m\ge1$ are integers.
Let $E_n=E_n(S)$.
\begin{itemize}
\item[(i)]
If $S$ begins with $\HH^k$, then 
$E_n(S|\HH^i) = E_n-E_{n-1}-\cdots-E_{n-i}$ for $i=1,\ldots,k$.
\item[(ii)]
If $S$ begins with $\HH^k\TT$, then $E_n(S|\HH^k\TT) = E_n-2E_{n-1}+E_{n-k-1}$.
\item[(iii)] 
If $S$ begins with $\HH^k\TT^\ell\HH$, then $E_n(S|\HH^k\TT^\ell\HH) = E_n-2E_{n-1}+E_{n-k-\ell}-E_{n-k-\ell-1}$.
\end{itemize}
\end{lemma}

\begin{proof}
(i) For $i=1,\ldots,k$, we have $E_n(S|\HH^{i-1}\TT) = E_{n-i}$ as the occurrence of $\HH^{i-1}\TT$ at the beginning does not affect when $S$ will occur, and this implies by induction that 
\[
E_n(S|\HH^i) = E_n(S|\HH^{i-1}) - E_n(S|\HH^{i-1}\TT) = E_n-E_{n-1}-\cdots-E_{n-i}.
\]  

(ii) We have $E_n(S|\HH^k\TT) = E_n(S|\HH^k) - E_n(S|\HH^{k+1})$ and $E_n(S|\HH^{k+1}) = E_{n-1}(S|\HH^k)$.
Since the result from (i) is still valid here, it follows that $E_n(S|\HH^k\TT) = E_n-2E_{n-1}+E_{n-k-1}$.

(iii) In this case we have an alternative way to compute $E_n(S|\HH^k\TT)$:
\begin{align*}
E_n(S|\HH^k\TT) =& \sum_{i=1}^\ell E_n(S|\HH^k\TT^i\HH) + E_n(S|\HH^k\TT^{\ell+1}) \\
=& \sum_{i=1}^{\ell-1} E_{n-k-i}(S|\HH) + E_n(S|\HH^k\TT^\ell\HH) + E_{n-k-\ell-1} \\
=& \sum_{i=1}^{\ell-1} (E_{n-k-i}- E_{n-k-i-1}) + E_n(S|\HH^k\TT^\ell\HH) + E_{n-k-\ell-1} \\
=& E_{n-k-1} - E_{n-k-\ell} + E_n(S|\HH^k\TT^\ell\HH) + E_{n-k-\ell-1}.
\end{align*}
Comparing this with (ii) gives the desired result.
\end{proof}

Recall that the \emph{Fibonacci number of order $k$} is defined by $F_n^k = 0$ for $n<k-1$, $F_n^k=1$ for $n=k-1$, and $F_n^k = F_{n-1}^k + F_{n-2}^k + \cdots + F_{n-k}^k$ for $n\ge k$.
We have $F_{n,1}=1$ for $n\ge0$, $F_{n,2}=F_n$ is the well-known \emph{Fibonacci number}~\cite[A000045]{OEIS}, $F_{n,3}$ is the \emph{tribonacci number}~\cite[A000073]{OEIS}, $F_{n,4}$ is the \emph{tetranacci number}~\cite[A000078]{OEIS}, and so on.

We define a variation of the Fibonacci number of order $k$ by $\overline F_n^k = 0$ for $n<k-1$, $\overline F_n^k = 1$ for $n=k-1$, and $\overline F_n^k = \overline F_{n-1}^k + \overline F_{n-2}^k + \cdots + \overline F_{n-k}^k + 1$ for $n\ge k$.
By a direct argument or using the generating functions given in Remark~\ref{rem1}, one can show that the above recurrence is equivalent to either $\overline F_n^k = 2\overline F_{n-1}^k -  \overline F_{n-k-1}^k$ or
\[
\overline F_n^k = \overline F_{n-1}^k + F_n^k
= \overline F_{n-2}^k + F_{n-1}^k + F_n^k 
= \cdots = F_0^k + F_1^k + \cdots + F_n^k.
\]
Special cases include $\overline F_n^1=n+1$, $\overline F_n^2 = F_{n+2}-1$~\cite[A000071]{OEIS}, $\overline F_n^3$~\cite[A008937]{OEIS}, $\overline F_n^4$~\cite[A107066]{OEIS}, and so on.

We now compute $E_n(S)$ when $S$ consists of at most two maximal runs of heads or tails.

\begin{proposition}\label{prop1}
For any integers $k,\ell\ge1$, we have $E_n(\HH^k) = F_{n-1}^k$ and $E_n(\HH^k\TT^\ell) = \overline F_{n-2}^{k+\ell-1}$.
\end{proposition}

\begin{proof}
First, let $E_n=E_n(S)$, where $S=\HH^k$.
Then $E_n=0$ if $n<k$ and $E_n=1$ if $n=k$.
For $n>k$ we have $0=E_n(S|\HH^k) = E_n-E_{n-1}-\cdots-E_{n-k}$ by Lemma~\ref{lem:HT'} (i).
Thus $E_n = F_{n-1}^k$. 

Next, let $E_n = E_n(S)$, where $S = \HH^k\TT^\ell$.
Then $E_n=0$ if $n< k+\ell$ and $E_n = 1$ if $n=k+\ell$.
Assume $n>k+\ell$ below.
By Lemma~\ref{lem:HT'}, we have $E_n = 2E_{n-1}-E_{n-k-1} +E_n(S|\HH^k\TT)$ and
\begin{align*}
E_n(S|\HH^k\TT) 
&= \sum_{i=1}^{\ell-1} E_n(S|\HH^k\TT^i\HH) + E_n(S|\HH^k\TT^\ell)
= \sum_{i=1}^{\ell-1} E_{n-k-i}(\HH)  \\
&= \sum_{i=1}^{\ell-1} (E_{n-k-i}-E_{n-k-i-1}) 
= E_{n-k-1} - E_{n-k-\ell}.
\end{align*}
Thus $E_n = 2E_{n-1}-E_{n-k-\ell}$.
It follows that $E_n = \overline F_{n-2}^{k+\ell-1}$.
\end{proof}

Combining Proposition~\ref{prop1} with Eq.~\eqref{eq:E_n} and Theorem~\ref{thm1} gives the following corollary.

\begin{corollary}\label{cor:id1}
For each integer $k\ge1$, we have
$\displaystyle \sum_{n=0}^\infty \frac{n F_{n-1}^k}{2^n} = 2^{k+1}-2$
and 
$\displaystyle \sum_{n=0}^\infty \frac{n \overline F_{n-2}^k}{2^n} = 2^{k+1}$.
\end{corollary}

\begin{proof}
Taking $S=\HH^k$ in Eq.~\eqref{eq:E_n} gives the first desired identity. 
Taking $S=\HH^k\TT^\ell$ in Eq.~\eqref{eq:E_n} gives 
$\displaystyle \sum_{n=0}^\infty \frac{n \overline F_{n-2}^{k+\ell-1}}{2^n} = E(\HH^k\TT^\ell) = 2^{k+\ell}$, which is equivalent to the second desired identity.
\end{proof}

\begin{remark}\label{rem1}
We can prove Corollary~\ref{cor:id1} by generating functions.
Let $f(x) = \sum_{n\ge0} F_{n-1}^k x^n$. Then
\begin{align*}
xf(x) + x^2f(x) + \cdots + x^{k} f(x) 
& = \sum_{n\ge k} \sum_{i=1}^k F_{n-i}^k x^{n+1} 
= \sum_{n\ge k} F_n^k x^{n} = f(x) - x^{k}.
\end{align*}
Thus $f(x) = \frac{x^{k}}{1-x-x^2-\cdots-x^k}$.
We can then substitute $x=\frac12$ into $xf'(x)$ and simplify the result by Lemma~\ref{lem:id} to obtain the first identity in Corollary~\ref{cor:id1}.
Similarly, $\overline f(x) = \sum_{n\ge 0} \overline F_{n-2}^k x^n$ satisfies
\[
x\overline f(x) + x^2\overline f(x) + \cdots + x^{k} \overline f(x) 
= \sum_{n \ge k} \sum_{i=1}^k \overline F_{n-i}^k x^{n+2} 
= \sum_{n\ge k} (\overline F_n^k-1) x^{n+2} = \overline f(x) - \frac{x^{k+1}}{1-x}.
\]
Thus $\overline f(x) = \frac{x^{k+1}}{(1-x)(1-x-x^2-\cdots-x^k)}
= \frac{x^{k+1}}{1-2x+x^{k+1}}$.
Substituting $x=\frac12$ into $xf'(x)$ gives the second identity in Corollary~\ref{cor:id1}.
\end{remark}

Next, we introduce a two-parameter generalization of both $F_n^k$ and $\overline F_n^k$.
Let $F_{n}^{k,m} = 0$ for $n<k+m-1$, $F_{n}^{k,m} = 1$ for $n=k+m-1$, and for $n\ge k+m$,
\[ F_{n}^{k,m} =
2F_{n-1}^{k,m}- F_{n-k-1}^{k,m} + \sum_{i=1}^{m} F_{n-k-i-1}^{k,m}. 
\]
For $k=0$ we have $F_n^{0,m} = F_{n+1}^{m+1}$ since
\[ F_n^{0,m} = 2F_{n-1}^{0,m} - F_{n-1}^{0,m} + \sum_{i=1}^{m} F_{n-i-1}^{0,m} 
= \sum_{i=1}^{m+1} F_{n-i}^{0,m}.
\]
For $m=0$ we have $F_n^{k,0} = 2F_{n-1}^{k,0} - F_{n-k-1}^{k,0}$, so $F_n^{k,0} = \overline F_{n}^{k}$.
The case $k=m=1$~\cite[A005314]{OEIS} and the case $(k,m)=(2,1)$~\cite[A059633]{OEIS} have also been studied before from other perspectives.

We can use $F_n^{k,m}$ to express $E_n(S)$ when $S$ consists of three maximal runs of heads or tails.

\begin{proposition}\label{prop2}
For any integers $k,\ell,m\ge1$, we have $E_n(\HH^k\TT^\ell\HH^m) = F_{n-2}^{k+\ell-1,m}$ if $m\le k$ and 
$E_n(\HH^k\TT^\ell\HH^m) = F_{n-2}^{\ell+m-1,k}$ if $m>k$.
\end{proposition}

\begin{proof}
Let $E_n=E_n(S)$, where $S=\HH^k\TT^\ell\HH^m$. 
We have $E_n=0$ for $n<k+\ell+m$ and $E_n=1$ for $n=k+\ell+m$.
It remains to show for $n>k+\ell+m$ that
\begin{equation}\label{eq:HTH}
E_n = 
\begin{cases}
2E_{n-1} - E_{n-k-\ell} + \sum_{i=1}^{m} E_{n-k-\ell-i}, &\text{if } m\le k; \\[.5em]
2E_{n-1} -E_{n-\ell-m} + \sum_{i=1}^k E_{n-\ell-m-i}, & \text{if } m>k.
\end{cases}
\end{equation} 
By Lemma~\ref{lem:HT} (iii), we have the following, where the last term is zero since $n>k+\ell+m$:
\[ 
E_n = 2E_{n-1} - E_{n-k-\ell} + E_{n-k-\ell-1} + \sum_{i=1}^{m-1} E_n(S|\HH^k\TT^\ell\HH^i\TT) + E_n(S|\HH^k\TT^\ell\HH^m).
\]
If $k\le i<m$ then 
$E_n(S|\HH^k\TT^\ell\HH^i\TT) = E_{n-\ell-i}(S|\HH^k\TT) = E_{n-\ell-i} - 2E_{n-\ell-i-1} + E_{n-k-\ell-i-1}$
by Lemma~\ref{lem:HT'}.
If $i<\min\{k,m\}$ then $E_n(S|\HH^k\TT^\ell\HH^i\TT) = E_{n-k-\ell-i-1}$.
Thus when $m\le k$ we have 
$E_n = 2E_{n-1} - E_{n-k-\ell} + \sum_{i=1}^{m} E_{n-k-\ell-i}$
and when $m>k$ we have
\begin{align*}
E_n &= 2E_{n-1} - E_{n-k-\ell} + \sum_{i=1}^{k} E_{n-k-\ell-i}
+ \sum_{i=k}^{m-1} (E_{n-\ell-i} - 2E_{n-\ell-i-1} + E_{n-k-\ell-i-1}) \\
&=  2E_{n-1} + \sum_{i=1}^{k} E_{n-k-\ell-i} - \sum_{i=k+1}^{m-1} E_{n-\ell-i} -2E_{n-\ell-m} + \sum_{i=k+1}^{m} E_{n-k-\ell-i} \\
&=  2E_{n-1} -2E_{n-\ell-m} + \sum_{i=m-k}^m E_{n-k-\ell-i}
=  2E_{n-1} -E_{n-\ell-m} + \sum_{i=1}^k E_{n-\ell-m-i}.
\qedhere
\end{align*}
\end{proof}

Proposition~\ref{prop2} suggests an identity involving $F_n^{k,m}$, which can be proved by generating functions to avoid the restriction on the relative sizes of $k$ and $m$.

\begin{corollary}\label{cor:id2}
For any integers $k,m\ge1$ we have $\displaystyle \sum_{n\ge1} \frac{n F_{n-2}^{k,m}}{2^n} = 2^{k+m+1}+2^{m+1}-2$.
\end{corollary}

\begin{proof}
By the definition of $F_n^{k,m}$, we have the generating function
\[ \sum_{n\ge0} F_{n-2}^{k,m} x^n 
= \frac{x^{k+m+1}}{1-2x+x^{k+1}-\sum_{i=1}^m x^{k+i+1}}.
\]
Evaluating $x$ times the derivative of the above generating function at $x=\frac12$ and using Lemma~\ref{lem:id} to simplify the result we obtain the desired identity. 
\end{proof}

Note that one can recover Corollary~\ref{cor:id1} from Corollary~\ref{cor:id2} by taking $k=0$ or $m=0$.

Next, we introduce another two-parameter generalization of the Fibonacci numbers by defining $\widetilde F_n^{k,m}=0$ for $n<k+m-1$, $\widetilde F_n^{k,m}=1$ for $n=k+m-1$, and for $n>k+m-1$, 
\[ 
\widetilde F_n^{k,m} = 2 \widetilde F_{n-1}^{k,m} - \widetilde F_{n-k-1}^{k,m} 
+ 2 \widetilde F_{n-k-2}^{k,m} - \widetilde F_{n-k-m-2}^{k,m}.
\]
We have $\widetilde F_n^{0,0} = F_{n+2}$, $\widetilde F_n^{k,0} = F_{n+1}^{k,1}$, and some other interesting special cases of $\widetilde F_n^{k,m}$, such as
$(k,m)=(0,1)$~\cite[A006053]{OEIS},
$(k,m)=(0,2)$~\cite[A158943]{OEIS},
and $(k,m)=(1,1)$~\cite[A112575]{OEIS}. 

To derive an identity involving $\widetilde F_n^{k,m}$, we compute $E_n(S)$ when $S$ consists of four maximal runs of heads or tails.

\begin{proposition}\label{prop3}
For any integers $k,\ell,m,d\ge1$, we have 
\[ E_n(\HH^k\TT^\ell\HH^m\TT^d) = 
\begin{cases}
\overline F_{n-2}^{k+\ell+m+d-1}, & \text{if } m<k \text{ or } d>\ell; \\
\widetilde F_{n-3}^{\ell+m-1, k+d-1}, & \text{if } m\ge k \text{ and } d\le \ell.
\end{cases} \]
\end{proposition}

\begin{proof}
Let $E_n=E_n(S)$, where $S=\HH^k\TT^\ell\HH^m\TT^d$. 
We have $E_n=0$ for $n<k+\ell+m+d$ and $E_n=1$ for $n=k+\ell+m+d$.
Assume $n>k+\ell+m+d$ below. 
We can use Lemma~\ref{lem:HT'} to evaluate $E_n$ in the same way as the proof of Proposition~\ref{prop2} and obtain the same result except that $E_n(S|\HH^k\TT^\ell\HH^m)$ may not be zero.
Thus we only need to find $E_n(S|\HH^k\TT^\ell\HH^m)$ and add it to the right hand side of Eq.~\eqref{eq:HTH}.
We begin with the following, where the last term is zero since $n>k+\ell+m+d$:
\[ 
E_n(S|\HH^k\TT^\ell\HH^m) = E_n(S|\HH^k\TT^\ell\HH^{m+1}) 
+ \sum_{i=1}^{d-1} E_n(S|\HH^k\TT^\ell\HH^m\TT^i\HH) + E_n(S|\HH^k\TT^\ell\HH^m\TT^d).
\]
If $m<k$ then using $E_n(S|\HH^i)=E_n-E_{n-1}-\cdots-E_{n-i}$ for $i=1,\ldots,k$ from Lemma~\ref{lem:HT'} we obtain
\begin{align*}
E_n(S|\HH^k\TT^\ell\HH^m) 
=& E_{n-k-\ell} - \sum_{i=1}^{m+1} E_{n-k-\ell-i} + \sum_{i=1}^{d-1} (E_{n-k-\ell-m-i} - E_{n-k-\ell-m-i-1})  \\
=& E_{n-k-\ell} - \sum_{i=1}^{m} E_{n-k-\ell-i} - E_{n-k-\ell-m-d},
\end{align*}
and adding this to the right hand side of Eq.~\eqref{eq:HTH} in the case $m<k$ gives $E_n = 2E_{n-1} - E_{n-k-\ell-m-d}$.
Similarly, if $m\ge k$ and $d\le\ell$ then we can add
\begin{align*}
E_n(S|\HH^k\TT^\ell\HH^m) 
&= E_{n-\ell-m-1}(S|\HH^k) + \sum_{i=1}^{d-1} E_{n-k-\ell-m-i}(S|\HH) \\
&= E_{n-\ell-m-1}-\sum_{i=2}^k E_{n-\ell-m-i} - E_{n-k-\ell-m-d}
\end{align*}
to the right hand side of Eq.~\eqref{eq:HTH} in the case $m\ge k$ (the two cases there coincide when $m=k$) gives
\[ 
E_n = 2E_{n-1} -E_{n-\ell-m} + 2E_{n-\ell-m-1} - E_{n-k-\ell-m-d}.
\]
Finally, if $m\ge k$ and $d>\ell$ then by Lemma~\ref{lem:HT'},
\begin{align*}
E_n(S|\HH^k\TT^\ell\HH^m) 
=& E_{n-\ell-m-1}(S|\HH^k) + \sum_{1\le i<d,\ i\ne\ell} E_{n-k-\ell-m-i}(S|\HH)
+ E_{n-\ell-m}(S|\HH^k\TT^\ell\HH) \\
=& E_{n-\ell-m-1}-\sum_{i=2}^k E_{n-\ell-m-i} 
- E_{n-k-\ell-m-\ell} + E_{n-k-\ell-m-\ell-1} \\
& - E_{n-k-\ell-m-d} + E_{n-\ell-m} - 2E_{n-\ell-m-1}+E_{n-\ell-m-k-\ell}-E_{n-\ell-m-k-\ell-1}  \\
=& E_{n-\ell-m} - E_{n-\ell-m-1}-\sum_{i=2}^k E_{n-\ell-m-i}
- E_{n-k-\ell-m-d},
\end{align*}
and adding this to the right hand side of Eq.~\eqref{eq:HTH} in the case $m\ge k$ gives
$E_n = 2E_{n-1} - E_{n-k-\ell-m-d}$.
The result follows.
\end{proof}

Proposition~\ref{prop3} suggests an identity involving $\widetilde F_n^{k,m}$, which can be proved by generating functions to avoid any restriction on the relative sizes of $k$ and $m$.

\begin{corollary}\label{cor:id3}
For any integers $k,m\ge1$ we have $\displaystyle \sum_{n\ge1} \frac{n \widetilde F_{n-3}^{k,m}}{2^n} = 2^{k+m+2}+2^{m+1}$.
\end{corollary}

\begin{proof}
By the definition of $F_n^{k,m}$, we have the generating function
\[ \sum_{n\ge0} \widetilde F_{n-3}^{k,m} x^n 
= \frac{x^{k+m+2}}{1-2x+x^{k+1}-2x^{k+2}+x^{k+m+2}}.
\]
The desired identity then follows from evaluating $x$ times the derivative of the above generating function at $x=\frac12$.
\end{proof}

Lastly, we study $E_n(S)$, where $S$ is a string of length $s$ alternating between heads and tails.
Define $G_n^s = 0$ for $n<s$, $G_n^s=1$ for $n=s$, and for $n>s$,
\[ G_n^s = 
\begin{cases}
\sum_{i=1}^k (2G_{n-2i+1}^s - G_{n-2i}^s), & \text{if } s=2k; \\[.5em]
\sum_{i=1}^{k-1} (2G_{n-2i+1}^s - G_{n-2i}^s) + G_{n-2k+1}^s, & \text{if } s=2k-1.
\end{cases} \]
Special cases include $G_n^2 = \widetilde F_{n-1}^{1,1}$~\cite[A112575]{OEIS} 
and $G_n^3=F_{n-2}^{1,1}$~\cite[A005314]{OEIS}.

\begin{proposition}\label{prop:alternate}
Let $S$ be an alternating string of heads and tails with length $s$.
Then $E_n(S) = G_n^s$ for all $n\ge0$.
\end{proposition}

\begin{proof}
Let $E_n = E_n(S)$.
Then $E_n=0$ for $n<s$ and $E_n=1$ for $n=s$. 
Assume $n>s$ below.
By symmetry, we may assume that $S$ begins with a heads.
If $s=2k$ then $S=(\HH\TT)^k$ and
\begin{align*}
E_n & = E_n(S|\TT) + E_n(S|\HH\HH) + E_n(S|\HH\TT^2) + E_n(S|\HH\TT\HH^2) + \cdots + E_n(S|(\HH\TT)^{k-1}\HH^2) + E_n(S|S) \\
&= E_{n-1} + E_{n-1}(S|\HH) + E_{n-3} + E_{n-3}(S|\HH) + \cdots + E_{n-2k+1} + E_{n-2k+1}(S|\HH) + \delta_{n,2k} \\
&= 2E_{n-1} - E_{n-2} + 2E_{n-3} - E_{n-4} + \cdots + 2E_{n-2k+1} - E_{n-2k}.
\end{align*}
Similarly, if $s=2k-1$ then $S=(\HH\TT)^{k-1}\HH$ and
\begin{align*}
E_n & = E_n(S|\TT) + E_n(S|\HH\HH) + E_n(S|\HH\TT^2) + E_n(S|\HH\TT\HH^2) + \cdots + E_n(S|(\HH\TT)^{k-1}\TT) + E_n(S|S) \\
&= E_{n-1} + E_{n-1}(S|\HH) + E_{n-3} + E_{n-3}(S|\HH) + \cdots + E_{n-2k-1} + \delta_{n,2k-1} \\
&= 2E_{n-1} - E_{n-2} + 2E_{n-3} - E_{n-4} + \cdots + E_{n-2k-1}.
\end{align*}
In either case we have $E_n=G_n^s$.
\end{proof}

Combining Proposition~\ref{prop:alternate} with Eq.~\eqref{eq:E_n} and Theorem~\ref{thm:alternate} gives the following corollary.

\begin{corollary}\label{cor:alternate}
For any integer $k\ge1$, 
$\displaystyle \sum_{n=0}^\infty \frac{nG_n^{2k}}{2^{n}} 
= \frac{2^{2k+2}-4}{3}$ and
$\displaystyle \sum_{n=0}^\infty \frac{nG_n^{2k-1}}{2^{n}} 
= \frac{2^{2k+1}-2}{3}$.
\end{corollary}

\begin{remark}\label{rem:alternate}
We can prove Corollary~\ref{cor:alternate} by generating functions.
For any integer $k\ge1$, we have
\[ \sum_{n=0}^\infty G_n^s x^n = 
\begin{cases}
\frac{x^{2k}}{1-\sum_{i=1}^k (2x^{2i-1}-x^{2i})} 
= \frac{x^{2k}(1-x^2)}{1-2x+2x^{2k+1}-x^{2k+2}} 
= \frac{x^{2k}(1+x)}{1-\sum_{i=1}^{2k} x^i+x^{2k+1}},
&\text{if } s = 2k; \\[1em]
\frac{x^{2k-1}}{1-\sum_{i=1}^{k-1} (2x^{2i-1}-x^{2i}) - x^{2k-1}} 
= \frac{x^{2k-1}(1-x^2)}{1-2x+x^{2k-1}-x^{2k}+x^{2k+1}} 
= \frac{x^{2k-1}(1+x)}{1-\sum_{i=1}^{2k-2} x^i - x^{2k}}, 
&\text{if } s = 2k-1.
\end{cases} \]
We then take the derivative, multiply by $x$ and substitute in $x=\frac12$ to obtain the desired identities.
\end{remark}

\section{Questions}\label{sec:questions}

In this paper, we determine the expected number of coin flips to produce a given finite string which either has at most four maximal runs of heads or tails or alternates between heads and tails; 
the result is always a sum of powers of $2$.
It would be nice to provide an intuitive explanation for this.
We also have the following conjecture for an arbitrary ending string based on our result.

\begin{conjecture}
Suppose $S\notin\{\HH^k,\TT^k\}$ is a string consisting of a total of $s$ heads and tails.
Then $E(S)$ equals $2^s$ possibly plus some lower positive powers of $2$, so asymptotically, $E(S)$ is about $2^s$.
\end{conjecture}

Moreover, we have $E(\HH^k\TT^\ell) = E(\HH^\ell\TT^k)$ by Theorem~\ref{thm1}, $E(\HH^k\TT^\ell\HH^m) = E(\HH^m\TT^\ell\HH^k)$ by Theorem~\ref{thm2}, and $E(\HH^k\TT^\ell\HH^m\TT^d) = E(\HH^d\TT^m\HH^\ell\TT^k)$ by Theorem~\ref{thm3}.
This leads to another conjecture.

\begin{conjecture}
If $S'$ is the reversal of an ending string $S$, then $E(S)=E(S')$.
\end{conjecture}

We do not have a rigorous proof for the above conjecture, but here is an intuitive argument: if one flips a coin a large number of times, then $E(S)$ is the average distance between an occurrence of $S$ and the next, but reading the outcomes backward, we have the same average distance between an occurrence of $S'$ and the next, so $E(S')=E(S)$.

One of our summation identities (the first one in Corollary~\ref{cor:id1}) partially overlaps with the sum $a_n = \sum_{k\ge1} \frac{k^n F_k}{2^{k+1}}$ studied by Benjamin, Neer, Otero, and Sellers~\cite{ProbFibSum} similarly via both a probabilistic view and a generating function approach.
It should be possible to generalize the above sum $a_n$ using higher order Fibonacci numbers or the two-parameter generalizations of Fibonacci numbers discussed in this paper.

Another way to extend our work is to study a similar game using a die instead of a coin.
For example, if $E$ is the expected number of times to throw a die with $c$ faces marked $1,\ldots, c$ to obtain a run of $k$ ones, then 
$\displaystyle E = \sum_{i=1}^k \frac{(c-1)(E+i)}{c^i} + \frac{k}{c^k}$, and solving this equation with the help of the formula $\displaystyle \sum_{i=1}^k \frac{i}{c^i} = \frac{c}{(c-1)^2} - \frac{(c-1)k+c}{(c-1)^2\cdot c^k}$ yields $\displaystyle E=\frac{c^{k+1}-c}{c-1}$. 
Taking $c=2$ recovers our result $E(\HH^k) = 2^{k+1}-2$ in Theorem~\ref{thm1}.

One could also try to extend work of Ekhad and Zeilberger~\cite{EkhadZeilberger} 
and work of Segert~\cite{Segert} by investigating how often those strings of heads and tails discussed in this paper will occur when flipping a coin a large number of times.
Some of the results in this paper might be helpful in a similar way as how comparing $E(\HH\HH)=6$ with $E(\HH\TT)=4$ at least intuitively explains why $\HH\TT$ outnumbers $\HH\HH$ in the long run.

\end{document}